\documentclass{proc-l}

\newtheorem{theorem}{Theorem}[section]

\newtheorem{cor}[theorem]{Corollary}
\newtheorem{proposition}[theorem]{Proposition}
\theoremstyle{definition}

\theoremstyle{remark}
\newtheorem{remark}[theorem]{Remark}

\numberwithin{equation}{section}

\def\pfaf{\mathrm{Pf}}
\def\sgn{\mathrm{sgn}}
\def\M{\mathcal{M}}
\def\cro{\textrm{cr}}

\newcommand\C{\mathbb{C}}
\newcommand\N{\mathbb{N}}

\def\L{\mathcal L}

\def\A{{\mathbf A}}
\def\B{{\mathbf B}}
\def\e{{\mathbf E}}
\def\f{{\mathbf F}}

\begin{document}
\title[ A quadratic  formula for basic hypergeometric series  ]{A quadratic  formula for basic hypergeometric series related to Askey-Wilson polynomials}
\author[V. J. W. Guo]{Victor J. W. Guo}
\address{Department of Mathematics, East China Normal University, Shanghai 200062,
 People's Republic of China}
\email{jwguo@math.ecnu.edu.cn}
\author[M. Ishikawa]{Masao Ishikawa}
\address{Department of Mathematics, Faculty of Education, University of the Ryukyus, Nishihara, Okinawa
901-0213, Japan}
\email{ishikawa@edu.u-ryukyu.ac.jp}
\author[H. Tagawa]{Hiroyuki Tagawa}
\address{Department of Mathematics, Faculty of Education, Wakayama University,
Sakaedani, Wakayama 640-8510, Japan}
\thanks{Partially supported by JSPS Grant-in-Aid for Scientific Research~(C)~23540017.}
\email{tagawa@math.edu.wakayama-u.ac.jp}
\author[J. Zeng]{Jiang Zeng}
\address{Universit\'e de Lyon; Universit\'e Lyon 1; Institut Camille
Jordan, UMR 5208 du CNRS; 43, boulevard du 11 novembre 1918,
F-69622 Villeurbanne Cedex, France}
\email{zeng@math.univ-lyon1.fr}
\thanks{Partially supported by CMIRA COOPERA 2012 de  la R\'egion Rh\^one-Alpes.}

\date{November 7, 2012 and, in revised form, July 17, 2013.}



\dedicatory{
Dedicated to Srinivasa Ramanujan on the occasion of
his 125th birth anniversary}

\commby{Sergei Suslov}

\maketitle

\begin{abstract}
We prove a general quadratic formula for basic hypergeometric series, from which simple proofs of  several recent determinant and Pfaffian formulas
are obtained.
A special case of the quadratic formula is actually related to
a Gram determinant formula for  Askey-Wilson polynomials.
We also  show how to derive a
  recent double-sum formula for the moments of Askey-Wilson polynomials from   Newton's
interpolation formula.
\end{abstract}

\section{Introduction}
Throughout this paper we assume that $q$ is a fixed number in $(0,\; 1)$.
A $q$-shifted factorial is defined by
\[
(a;q)_\infty=\prod_{k=0}^{\infty}(1-aq^{k}),\quad\text{and}
\quad (a;q)_n=\frac{(a;q)_\infty}{(aq^n;q)_\infty}, \quad\text{for}
\quad n\in\mathbb{Z}.
\]
Following Gasper and Rhaman~\cite{GR} we shall use  the abbreviated notation
$$
(a_1,a_2,\ldots,a_r;q)_n
=(a_1;q)_n (a_2;q)_n \cdots(a_r;q)_n,
\quad\text{for}\quad n\in\mathbb{Z}.
$$
A  {\it basic hypergeometric series with $r$ numerators and $s$ denominators}
is then defined by
\begin{align*}
{}_{r}\phi_{s}
\left[\,
\begin{matrix}
{a_{1},a_{2},\dots,a_{r}}\\
{b_{1},\dots,b_{s}}
\end{matrix}
\,;\,q,z
\,\right]
=
\sum_{n=0}^{\infty}
\frac{(a_{1},a_{2},\dots,a_{r};q)_{n}}{(q,b_{1},\dots,b_{s};q)_{n}}
\left((-1)^nq^{\binom{n}{2}}\right)^{1+s-r}z^n.
\end{align*}
The Askey-Wilson polynomials $p_{n}(x;a,b,c,d;q)$ ($n\in \N$) are the ${}_{4}\phi_{3}$ polynomials~\cite{AW,GR}
\begin{align*}
\frac{(ab,ac,ad;q)_{n}}{a^{n}}
{}_{4}\phi_{3}\left[\!\!\! \begin{array}{c}
q^{-n},\, abcdq^{n-1},\,ae^{I\theta},\, ae^{-I\theta}\\
ab,\, ac,\, ad\end{array}\!\!\!;q,q\right],
\end{align*}
where $x=\cos\theta$ and $I$ is a complex number such that $I^2=-1$.

Our main result is the following quadratic formula for basic hypergeometric series, 
which was discovered by applying 
Desnanot--Jacobi adjoint matrix theorem  to compute some determinants of Mehta-Wang type~\cite{Meh:Wan, Nis, Kratt02, GITZ13}
or  deformed Gram determinants of orthogonal  polynomials~\cite{Wilson}.

\begin{theorem}
\label{GITZ}
Let $r,s\geq 0$, $a,b,c,d,q\in \C$, $\e_r=(e_1,e_2,\dots,e_r)\in \C^r$, $\f_s=(f_1,f_2,\dots,f_s)\in \C^s$.
Then we have
\begin{multline}\label{main}
(a-b)(a-c)(bc-d)(1-d)\\
\times
{}_{r+4}\phi_{s+3}
\left[
\begin{matrix}
a^{-1}bc,bcq^{-2},c,dq^{-1},\e_r\cr
aq^{-1},bq^{-1},bcd^{-1},\f_s\cr
\end{matrix};
q,z
\right]
{}_{r+4}\phi_{s+3}
\left[
\begin{matrix}
a^{-1}bc,bc,c,dq,\e_r q\cr
aq,bq,bcd^{-1},\f_s q\cr
\end{matrix};
q,q^{s-r}z
\right]\\
=
(a-d)
(1-b)
(1-c)
(bc-ad)\\
\quad\times
{}_{r+4}\phi_{s+3}
\left[
\begin{matrix}
a^{-1}bc,bcq^{-2},cq^{-1},d,\e_r\cr
aq^{-1},b,bcd^{-1}q^{-1},\f_s\cr
\end{matrix};
q,z
\right]
{}_{r+4}\phi_{s+3}
\left[
\begin{matrix}
a^{-1}bc,bc,cq,d,\e_r q\cr
aq,b,bcd^{-1}q,\f_s q\cr
\end{matrix};
q,q^{s-r}z
\right]
\\
\quad
-
(1-a)
(b-d)
(c-d)
(a-bc)\\
\quad
\times
{}_{r+4}\phi_{s+3}
\left[
\begin{matrix}
a^{-1}bcq^{-1},bcq^{-2},c,d,\e_r \cr
a,bq^{-1},bcd^{-1}q^{-1},\f_s \cr
\end{matrix};
q,z
\right]
{}_{r+4}\phi_{s+3}
\left[
\begin{matrix}
a^{-1}bcq,bc,c,d,\e_r q\cr
a,bq,bcd^{-1}q,\f_s q\cr
\end{matrix};
q,q^{s-r}z
\right].
\end{multline}
\end{theorem}

Let $s=r$, $d=a_0$, $c=a_1$, $b=0$, $a=b_1$, $e_1=f_1=0$, $e_i=a_i$ and $f_j=b_j$ ($2\leq j\leq r$)  we get the following 
result  by shifting $r$ to $r-3$.
\begin{cor}\label{THM}For $r\geq 1$, there holds
\begin{multline}\label{eq:GZ}
 (a_0-{1})(a_1-b_1)\, \\
 \times _{r+1}\phi_{r}\left[\!\!\!\begin{array}{c}
a_0/q,a_1,a_2,\ldots,a_{r}\\
b_1/q,b_2,\ldots,b_{r}
\end{array}\!\!\!;q,z
\right]{}
_{r+1}\phi_{r}\left[\!\!\!\begin{array}{c}
a_0q,a_1,a_2q\ldots,a_{r}q\\
b_1q,b_2q,\ldots,b_{r}q
\end{array}\!\!\!;q,z
\right]  \\
=(a_0-a_1)(1-b_1)\, \\
\quad\times
_{r+1}\phi_{r}\left[\!\!\!\begin{array}{c}
a_0,a_1,a_2,\ldots,a_{r}\\
b_1,b_2,\ldots,b_{r}
\end{array}\!\!;q,z
\right]{}
_{r+1}\phi_{r}\left[\!\!\!\begin{array}{c}
a_0,a_1,a_2q\ldots,a_{r}q\\
b_1,b_2q,\ldots,b_{r}q
\end{array}\!\!\!;q,z
\right]\,  \\
\quad-(1-a_1)(a_0-b_1)\,  \\
\quad\times{} _{r+1}\phi_{r}\left[\!\!\begin{array}{c}
a_0,a_1/q,a_2,\ldots,a_{r}\\
b_1/q,b_2,\ldots,b_{r}
\end{array}\!\!\!;q,z
\right]{}
_{r+1}\phi_{r}\left[\!\!\!\begin{array}{c}
a_0,a_1q,a_2q\ldots,a_{r}q\\
b_1q,b_2q,\ldots,b_{r}q
\end{array}\!\!\!;q,z
\right]. 
\end{multline}
\end{cor}

Taking $r=3$,  $z=q$, $a_{0}=q^{-n+1}$,
$a_{1}=abcdq^{n-1}$, $a_{2}=ae^{i\theta}$, $a_{3}=ae^{-i\theta}$,
$b_{1}=abq$,  $b_{2}=ac$ and $b_{3}=ad$ in Corollary~\ref{THM},
we obtain the following quadratic relation for Askey-Wilson polynomials.

\begin{cor}\label{awq} Let $n$ be a positive integer. There holds
\begin{multline}\label{eq:conj}
 ab(1-q^{n-1})(1-cdq^{n-2})p_{n}(x;a,b,c,d;q)p_{n-2}(x;aq,bq,c,d;q) \\
=(1-abq^{n-1})(1-abcdq^{n-1}) p_{n-1}(x;a,b,c,d;q) p_{n-1}(x;aq,bq,c,d;q)\\
{} -(1-ab)(1-abcdq^{2n-2}) p_{n-1}(x;aq,b,c,d;q) p_{n-1}(x;a,bq,c,d;q). \hskip 5.5mm
\end{multline}
\end{cor}


We shall prove Theorem \ref{GITZ} in the next section.  In Section~3, 
we  use  the Desnanot-Jacobi adjoint matrix theorem and  Corollary~\ref{awq}
to derive a determinant formula for
Askey-Wilson polynomials (cf. Theorem~\ref{newthm}), which turns out to be a generalization of several recent 
determinant and Pffafian evaluations in~\cite{Meh:Wan,CK,Kratt02,ITZ09,ITZ10}.
In Section~4, we connect  the determinant formula in  Theorem~\ref{newthm} to  a Gram determinant
formula  of  Askey-Wilson  polynomials and show how to derive a recent double--sum formula for the moments of Askey-Wilson polynomials
in~\cite{CSSW,IR} from Newton's  interpolation formula.

\section{Proof of Theorem \ref{GITZ}}
For $0\leq k\leq n$, we write
\begin{align*}
A_k&=
(a-b)
(a-c)
(bc-d)
(1-d)\alpha_k\alpha_{n-k+1}
\\&\times
\frac{
(a^{-1}bc,bcq^{-2},c,dq^{-1},\e_r;q)_k
(a^{-1}bc,bc,c,dq,\e_rq;q)_{n-k}
}
{
(q,aq^{-1},bq^{-1},bcd^{-1},\f_s;q)_k
(q,aq,bq,bcd^{-1},\f_sq;q)_{n-k}
},\\
B_k&=
(a-d)
(1-b)
(1-c)
(bc-ad)\alpha_k\alpha_{n-k+1}
\\&\times
\frac{
(a^{-1}bc,bcq^{-2},cq^{-1},d,\e_r;q)_k
(a^{-1}bc,bc,cq,d,\e_rq;q)_{n-k}
}
{
(q,aq^{-1},b,bcd^{-1}q^{-1},\f_s;q)_k
(q,aq,b,bcd^{-1}q,\f_sq;q)_{n-k}
},\\
C_k&=
(1-a)
(b-d)
(c-d)
(a-bc)\alpha_k\alpha_{n-k+1}
\\&\times
\frac{
(a^{-1}bcq^{-1},bcq^{-2},c,d,\e_r;q)_k
(a^{-1}bcq,bc,c,d,\e_rq;q)_{n-k}
}
{
(q,a,bq^{-1},bcd^{-1}q^{-1},\f_s;q)_k
(q,a,bq,bcd^{-1}q,\f_sq;q)_{n-k}
},
\end{align*}
where  $\alpha_k=\left\{(-1)^kq^{\frac{k(k-1)}{2}}\right\}^{s-r}$ for $k\geq 0$.
Equating the coefficients of $z^n$ on both sides of \eqref{main}
yields  the equivalent identity
\begin{align}\label{eq:sums}
\sum_{k=0}^{n}(A_{k}-B_{k}+C_{k})=0.
\end{align}
The key point  to prove \eqref{eq:sums} is the observation that
\begin{align}\label{eq:6terms}
A_{k}-B_{k}+C_{k}+\bigl(A_{n-k+1}-B_{n-k+1}+C_{n-k+1}\bigr)=0
\end{align}
for $0\leq k\leq n+1$,
where $A_{n+1}=B_{n+1}=C_{n+1}=0$. Indeed, 
summing \eqref{eq:6terms} over $k$ from $0$ to $n+1$ on both sides yields  immediately  \eqref{eq:sums}.

To prove \eqref{eq:6terms}  we start from the  identity 
\begin{multline}
(a - b) (a - c) (d - x) (b c - d x) 
(x - a y) \\
\times (x - b y) (x - c y) (y - d z) 
(a x - b c y) (d y - b c z)
\\
- (a - d) (b - x) (c - x) 
(a d - b c) 
(x - a y)\\ 
\times  (y - b z) (y - c z)
 (x - d y) 
(a x - b c y) (d x - b c y) \\
+
(b - d) (c - d) (a - x) 
(a x - b c) 
(y - a z)\\
\times  (x - b y) (x - c y) (x - d y) 
(a y - b c z) (d x - b c y)\\
=x y(a - b)(a - c)(a - d)(b - d)(c - d)\\
\times (1-y)(a d - b c)
(x - b c z) (x^2 - b c y)(y^2 - x z),\label{eqkxyz}
\end{multline}
which can be  easily checked either by hands or by Maple.
Replacing $(x,y,z)$ by $(q,q^k,q^n)$ in \eqref{eqkxyz}
and multiplying both sides of the resulting identity by
\[
\frac{
(c,d,a^{-1}bc;q)_{k-1}
(bcq^{-2},\e_r;q)_{k}
(c,d,a^{-1}bc,bc,\e_rq;q)_{n-k}
}{
(aq^{-1},bq^{-1},bcd^{-1}q^{-1};q)_{k+1}
(q,\f_s;q)_{k}
(aq,bq,bcd^{-1}q,q,\f_sq;q)_{n-k}
},
\]
we obtain
\begin{align}\label{eqkkare}
A_k-B_k+C_k
=(q^{n-k+1}-q^k)G_kG_{n-k+1}\Xi,
\end{align}
where 
\[
G_k=\frac{(1-q^k)(1-bcq^{k-2})
(a^{-1}bc,c,d;q)_{k-1}(bc;q)_{k-2}(\e_r;q)_k\alpha_k}
{(a,b,bcd^{-1},q;q)_k(\f_s;q)_k},
\]
and 
\begin{multline}\nonumber
\Xi=(a-b)(a-c)(a-d)(b-d)(c-d)(ad-bc)(1-bcq^{n-1})\\
\quad\times
\frac{
(1-a)(1-b)(1-bcd^{-1})(1-bcq^{-2})(1-bcq^{-1})(\f_s;q)_1
}
{
adq^2
(1-aq^{-1})(1-bq^{-1})(1-bcd^{-1}q^{-1})(\e_r;q)_1
},
\end{multline}
which  is independent of $k$.
Clearly  \eqref{eqkkare} implies \eqref{eq:6terms}.

\section{Application to determinant  and Pfaffian evaluation}
Given a matrix  $M$,
if $i_{1},\dots,i_{r}$ (resp. $j_{1},\dots,j_{r}$)
are  row (resp. column) indices,
 we denote by $M_{i_{1},\dots,i_{r}}^{j_{1},\dots,j_{r}}$ the matrix that remains when the rows $i_{1}, \ldots, i_{r}$ and columns $j_{1}, \ldots, j_{r}$ are deleted.
Let $n\geq 2$ and $M$  be  an $n\times n$ matrix.  Then the Desnanot-Jacobi adjoint matrix theorem \cite[Lemma~7.7]{Aigner} reads
\begin{align}
\det M\, \det M^{1 \; n}_{1\; n}=\det M^{1}_{1}\,\det M^{n}_{n}
-\det M^{1}_{n}\,\det M^{n}_{1}, \label{eq:Desnanot-Jacobi}
\end{align}
where we set $\det M^{1 \; n}_{1\; n}=1$ if $n=2$.

\begin{theorem}\label{newthm} For $n\geq 1$, $0\leq i\leq n-1$ and $1\leq j\leq n$, let
\begin{align}\label{defB}
 B_{i,j}=& \frac{(ab ; q)_{i+j-1} (-bq^{-1+j})}{ (abcd ; q)_{i+j} }\\
& \times \bigl[c + d -2x+ (1- cd)(aq^{i} + bq^{j-1}) - ab(c+d-2cdx)q^{i+j-1}\bigr].\nonumber
\end{align}
Then
\begin{align}\label{eq:det}
\det ( B_{i,j} )_{0 \le i \le n-1, 1 \le j \le n}
 = D_n(a,b)\cdot p_n(x ; a,b,c,d ; q),
\end{align}
where
\begin{align}\label{defD}
 D_n(a,b)=a^{n(n-1)/2} b^{n(n+1)/2} q^{n(n-1)(2n-1)/6}
 \prod_{i=0}^{n-1}
 \frac{ (ab, cd, q;q)_{i}}{ (abcd ; q)_{n+i} }.
\end{align}
\end{theorem}
\begin{proof}
By \eqref{defD} we have
\begin{align*}
{D_n(a,b)}/{D_{n-1}(a,b)}
&=a^{n-1}b^nq^{(n-1)^2}\frac{(ab,cd;q)_{n-1}(q;q)_{n-1}}{(abcdq^{n-1};q)_{n}(abcd;q)_{2n-2}},\\
{D_{n}(aq,b)}/{D_{n}(a,b)}
&=q^{n(n-1)/2}\frac{(abq;q)_{n-1}}{(abcdq^n;q)_{n}}
\frac{(1-abcd)^{n}}{(1-ab)^{n-1}}.
\end{align*}
Therefore
\begin{align}\label{D1}
\frac{D_n(a,b)/D_{n-1}(a,b)}{D_{n-1}(aq,bq)/D_{n-2}(aq,bq)}=\frac{ab(ab;q)_2(1-cdq^{n-2})(1-q^{n-1})}{(1-abq^{n-1})(1-abcdq^{n-1})(abcd;q)_2},
\end{align}
and
\begin{multline}\label{D2}
 \frac{D_{n-1}(aq,b)/D_{n-1}(a,b)}{D_{n-1}(aq,bq)/D_{n-1}(a,bq)}\\
=\frac{1-abq}{1-abq^{n-1}}
\frac{1-abcdq^{2n-2}}{1-abcdq^{n-1}}
\Bigl(\frac{1-abq}{1-ab}\Bigr)^{n-2}
\Bigl(\frac{1-abcd}{1-abcdq}\Bigr)^{n-1}.
\end{multline}
Let $M_{n}(a,b): =\det ( B_{i,j} )_{0 \le i \le n-1, 1 \le j \le n}$.
For $n=1$,   formula~\eqref{eq:det}  is obvious. Assume that $n\geq 2$.
Applying \eqref{eq:Desnanot-Jacobi} to the determinant in \eqref{eq:det} we obtain
\begin{multline*}
M_{n}(a,b)\,M_{n-2}(aq, bq)
=\frac{(ab;q)_{2}}{(abcd;q)_{2}} M_{n-1}(a,b)\, M_{n-1}(aq, bq)\\
-\Bigl( \frac{1-ab}{1-abcd} \Bigr)^{n}\Bigl( \frac{1-abcdq}{1-abq} \Bigr)^{n-2}
M_{n-1}(aq,b)\, M_{n-1}(a, bq).
\end{multline*}
It suffices to show that if we substitute $M_n(a,b)$ by
 $D_n(a,b)p_n(x;a,b,c,d;q)$ the above identity still holds,  i.e., for $n\geq 2$,
\begin{multline}\label{eq:key}
D_n(a,b)D_{n-2}(aq,bq)p_n(x;a,b,c,d;q) p_{n-2}(x;aq,bq,c,d;q)\\
=\frac{(ab;q)_{2}}{(abcd;q)_{2}}  D_{n-1}(a,b)D_{n-1}(aq,bq)p_{n-1}(x; a,b,c,d;q) p_{n-1}(x;aq,bq,c,d;q)\\
\quad{}\ -\Bigl( \frac{1-ab}{1-abcd} \Bigr)^{n}\Bigl( \frac{1-abcdq}{1-abq} \Bigr)^{n-2}
 D_{n-1}(aq,b)D_{n-1}(a,bq)\\
 \qquad{}\ \times p_{n-1}(x;a,bq,c,d;q) p_{n-1}(x; aq,b,c,d;q).
\end{multline}
Dividing the two sides of \eqref{eq:key}  by $D_{n-1}(a,b)D_{n-1}(aq,bq)$ and
applying the two  identities \eqref{D1} and \eqref{D2}, we see  that   \eqref{eq:key}
is exactly   the quadratic formula \eqref{eq:conj}.
 The result then follows by induction on $n$.
\end{proof}

The following formula  is  an extension of
Nishizawa's $q$-analogue of  Mehta-Wang's formula~\cite{Meh:Wan, Nis, GITZ13}.
\begin{cor} For $n\geq 1$,  there holds
\begin{multline}\label{eq:ITZ1}
\det\Bigl((q^{i-1}-cq^{j-1})\frac{(aq;q)_{i+j-2}}{(abq^{2};q)_{i+j-2}}\Bigr)_{1\leq i,j\leq n}\\
=(-1)^{n}a^{\frac{n(n-3)}{2}}q^{\frac{n(n+1)(2n-5)}{6}}(abcq;q^{2})_{n}
\prod_{k=1}^{n}\frac{(q;q)_{k-1}(aq;q)_{k}(bq;q)_{k-2}}{(abq^{2};q)_{k+n-2}} \\
\times {}_{4}\phi_{3}\left[\!\!\!\begin{array}{c}
q^{-n},\, \,abq^{n},\,(acq)^{\frac{1}{2}},\,-(acq)^{\frac{1}{2}}\\
aq,\, (abcq)^{\frac{1}{2}},\, -(abcq)^{\frac{1}{2}}
\end{array}\!\!\!;q,q\right].   
\end{multline}
\end{cor}
\begin{proof}
Since the  Askey-Wilson polynomials are  symmetric on $a$ and $b$, in  \eqref{eq:det}
replacing  $p_n(x;a,b,c,d;q)$  by $p_n(x ; b,a,c,d ; q)$ and
making  the following substitution:
$$
x\leftarrow 0,\quad a\leftarrow (aq/c)^{1/2}I, \quad b\leftarrow  -(acq)^{1/2}I,\quad \quad c\leftarrow  b^{1/2}I,\quad d\leftarrow  -b^{1/2}I
$$
where $I^2=-1$, gives \eqref{eq:ITZ1} as $(bq;q)_{-1}=1/(1-b)$.
\end{proof}

%
%
%

If $c=1$, we can sum the  $_{4}\phi_{3}$ in \eqref{eq:ITZ1}  by
Andrews'  terminating $q$-analogue of Watson's formula \cite[II.17]{GR}
\begin{align*} 
 {}_{4}\phi_{3}\left[\!\!\!\begin{array}{c}
q^{-n},\, a^{2}q^{n+1},\,b,\, -b\\
aq,\, -aq,\, b^{2}\end{array}\!\!\!;q,q\right]
=\begin{cases}
0,&\text{if $n$ is odd},\\[5pt]
\displaystyle\frac{b^{n}(q, a^{2}q^{2}/b^{2}; q^{2})_{n/2}}
{(a^{2}q^{2},\, b^{2}q; q^{2})_{n/2}},& \text{if $n$ is even},
\end{cases}
\end{align*}
and  deduce  the following  result,  which was  first proved in ~\cite{ITZ10}, and is also
a  $q$-analogue of \cite[Theorem 6]{Kratt02}.
\begin{cor}\label{cor32} For $m\geq 1$, there holds
\begin{align*}
&\hskip -2mm \det\Bigl((q^{i-1}-q^{j-1})\frac{(aq;q)_{i+j-2}}{(abq^{2};q)_{i+j-2}}
\Bigr)_{1\leq i,j\leq 2m}\\
&=a^{2m(m-1)}q^{\frac{2m(m-1)(4m+1)}{3}}
\prod_{k=1}^{m}\Biggl(\frac{(q, aq;q)_{2k-1} (bq;q)_{2k-2}}{(abq^{2}; q)_{2(k+m)-3}}  \Biggr)^{2}.
\end{align*}
\end{cor}


 Recall \cite{Aigner} that the Pfaffian of a skew-symmetric matrix $A=(A_{i,j})_{1\leq i,j\leq 2m}$  is defined by
 \begin{equation}\label{eq:def}
 \pfaf A=\sum_{\pi\in \M[1,\ldots, 2m]}\sgn\,\pi \prod_{i<j\atop i,j \text{ matched in } \pi}A_{i,j}.
\end{equation}
Here  $\M[a,\ldots, b]$ is the set of perfect matchings of the complete graph on $\{a,\ldots, b\}$ for any nonnegative integers $a$ and $b$ such that $a<b$, and
 $\sgn\,\pi=(-1)^{\cro\,\pi}$,   where $\cro\,\pi$ is the number of matched pairs $(i,\, j)$ and $(i',\;j')$ in $\pi$ such that $i<i'<j<j'$.
The following result was first proved by Ishikawa et al.~\cite{ITZ10}.

 \begin{cor} For $m\geq 1$,  there holds
\begin{multline*}
\pfaf\biggl((q^{i-1}-q^{j-1})\frac{(aq;q)_{i+j-2}}{(abq^{2};q)_{i+j-2}}
\biggr)_{1\leq i,j\leq 2m}\\
=a^{m(m-1)}q^{\frac{m(m-1)(4m+1)}{3}}
\prod_{k=1}^{m}\frac{(q, aq;q)_{2k-1} (bq;q)_{2k-2}}{(abq^{2}; q)_{2(k+m)-3}}.
\end{multline*}
\end{cor}
\begin{proof}
As the square of the  Pfaffian of any skew-symmetric matrix is its determinant,
we derive from Corollary \ref{cor32} that
\begin{multline}
\pfaf\biggl((q^{i-1}-q^{j-1})\frac{(aq;q)_{i+j-2}}{(abq^{2};q)_{i+j-2}}
\biggr)_{1\leq i,j\leq 2m}  \label{eq:key1} \\
=\varepsilon_{m} a^{m(m-1)}q^{\frac{m(m-1)(4m+1)}{3}}
\prod_{k=1}^{m}\frac{(q, aq;q)_{2k-1} (bq;q)_{2k-2}}{(abq^{2}; q)_{2(k+m)-3}}, 
\end{multline}
where $\varepsilon_{m}^2=1$. By \eqref{eq:key1},  the factor $\varepsilon_{m}$ is a rational function of $a$, $b$ and $q$,
and  only takes  values $1$ or $-1$. Hence, for fixed $m$,  we must have $\varepsilon_m=1$  or
$\varepsilon_m=-1$ regardless the values of  $a$, $b$ and $q$.
It remains  to  show that $\varepsilon_{m}=1$ for all $m\geq 1$.  Obviously we have $\varepsilon_1=1$. Suppose that $m\geq 2$.
Taking $b=0$ and replacing $a$ by $q^{a-1}$  the  identity \eqref{eq:key1} reduces to
\begin{align}
&\hskip -2mm \pfaf\bigl((q^{i}-q^{j}){(q^{a};q)_{i+j}}
\bigr)_{0\leq i,j\leq 2m-1} \label{eq:key2} \\
&=\varepsilon_{m} q^{m(m-1)(a-1)+\frac{m(m-1)(4m+1)}{3}}
\prod_{k=1}^{m}{(q, q^{a};q)_{2k-1}}.  \nonumber
\end{align}
Using  the $q$-gamma function~\cite[p. 20]{GR}
$$
\Gamma_{q}(x)=\frac{(q;q)_{\infty}}{(q^{x};q)_{\infty}} (1-q)^{1-x},\qquad 0<q<1,
$$
we can rewrite \eqref{eq:key2} as
\begin{align}\label{eq:key3}
&\hskip -2mm
\pfaf\left(q^{i}[j-i]_q\Gamma_{q}(a+i+j)
\right)_{0\leq i,j\leq 2m-1}\\
&=\varepsilon_{m}q^{m(m-1)(a-1)+\frac{m(m-1)(4m+1)}{3}}
\prod_{k=1}^{m}[2k-1]_{q}!\, \Gamma_{q}(a+2k-1),  \nonumber
\end{align}
where $[n]_{q}!=[1]_q[2]_q\cdots [n]_q$ and $[k]_{q}=(1-q^{k})/(1-q)$.
If we  multiply both sides of \eqref{eq:key3} by $[a+1]_{q}$
and let $a$ tend to $-1$, then,   the right-hand side becomes
\begin{align}
\varepsilon_{m} q^{\frac{m(m-1)(4m-5)}{3}}
\prod_{k=1}^{m}[2k-1]_{q}! \prod_{k=1}^{m-1}\Gamma_{q}(2k).\label{eq:key4}
\end{align}
On the other hand, by \eqref{eq:def},  the left-hand side can be written  as
\begin{align}
\sum_{\pi\in \M[0,\ldots, 2m-1]}\sgn\, \pi \lim_{a\to -1}\bigl
[a+1]_{q}\prod_{i<j\atop i, j  \textrm{mached in} \,\pi}q^{i}[j-i]_q\Gamma_{q}(a+i+j).\label{eq:key5}
\end{align}
In this sum,   matchings $\pi$ for which all matched pairs $i,\, j$ satisfy $i+j>1$ will not contribute, because the corresponding summands vanish.
Therefore, the survival matchings must match $0$ and $1$ and the sum in \eqref{eq:key5} reduces to
\begin{align*}
&\hskip -10mm
\sum_{\pi'\in \M[2,\ldots, 2m-1]}\sgn\, \pi' \prod_{i<j\atop i, j
\textrm{mached in} \,\pi'}q^{i}[j-i]_q\Gamma_{q}(i+j-1)\\
&=\pfaf\left(q^{i}[j-i]_q\Gamma_{q}(i+j-1)\right)_{2\leq i,j\leq 2m-1}\\[5pt]
&=\pfaf\left(q^{i+2} [j-i]_q\Gamma_{q}(i+j+3)
\right)_{0\leq i,j\leq 2m-3}\\[5pt]
&=\varepsilon_{m-1} q^{\frac{m(m-1)(4m-5)}{3}}
\prod_{k=1}^{m}[2k-1]_{q}!  \prod_{k=1}^{m-1}\Gamma_{q}(2k).
\end{align*}
Comparing with \eqref{eq:key4} we  see that $\varepsilon_{m}=\varepsilon_{m-1}=\cdots =\varepsilon_{1}=1$.
\end{proof}

\begin{remark}
Except the trivial but crucial point that $\varepsilon_m$ is independent of $a,b$ and $q$, the above proof is a $q$-adaptation  of Ciucu and Krattenthaler's proof  \cite{CK} for  the
$q\to 1$ case of \eqref{eq:key3}:
 \begin{align}\label{eq:CK}
 \pfaf((j-i)\Gamma(a+i+j))_{0\leq i,j\leq 2m-1}=\prod_{k=1}^{m}(2k-1)! \Gamma(a+2k-1).
 \end{align}
 As we have shown that $\varepsilon_m$ is actually independent of $q$,  we could also reduce the proof directly to \eqref{eq:CK} by taking the limit
 $q\to 1$ in \eqref{eq:key3}.
 \end{remark}
 \section{Link to  Gram determinants of Askey-Wilson polynomials}
  In this section we show that the determinant formula
 in Theorem~\ref{newthm} is actually related to 
a Gram determinant for the Askey-Wilson orthogonal polynomials~(see \cite{Wilson}). This permits to
  enlighten  the origin of the peculiar matrix coefficients \eqref{defB}.
  
Let $\{p_n(x)\}$ be a sequence of orthogonal polynomials  with respect to a measure $d\mu$, and $\{\phi_k\}$ and $\{\psi_k\}$ be two sequences of polynomials such that
$\phi_k$ and $\psi_k$ are of exact degree $k$. Then
\begin{align}\label{gramwilson}
 \left|
\begin{matrix}
    \mu_{0,0}&\mu_{0,1}& \cdots & \mu_{0,n} \\
   \mu_{1,0}& \mu_{1,1}& \cdots & \mu_{1,n} \\
    \vdots               & \vdots               &        & \vdots               \\
   \mu_{n-1,0}& \mu_{n-1,1}& \cdots & \mu_{n-1,n} \\
    \phi_0(x)&   \phi_1(x)& \cdots &  \phi_n(x)\\
\end{matrix} \right|=C\cdot p_n(x),
\end{align}
where  $\mu_{i,j}=\int \psi_i(x)\phi_j(x)d\mu$ and $C$ represents a factor of normalization.

Note  that the three variables $x, \theta, z$ are related each other as follows:
$$
z = e^{I \theta},\quad
 x = \cos \theta,\quad x=\frac{z+z^{-1}}{2}.
$$
For $|a|, |b|, |c|, |d|<1$ and  $n\geq 0$,  let $h_n:=h_n(a,b,c,d,q)$ with
\begin{align*}
h_0&=\frac{(abcd;q)_\infty}{(q,ab,ac, ad, bc, bd, cd;q)_\infty},\\
h_n&=h_0\frac{(1-q^{n-1}abcd)(q,ab, ac,ad,bc,bd,cd;q)_n}{(1-q^{2n-1}abcd)(abcd;q)_n}.
\end{align*}
Then, the orthogonality of Askey-Wilson polynomials reads
\begin{align}\label{orth}
\oint_{\mathcal C}p_{m}(\cos\theta;a,b,c,d;q)p_{n}(\cos\theta;a,b,c,d;q)
 w(\cos\theta)\frac{dz}{4\pi iz}=\frac{h_n}{h_0} \delta_{mn},
\end{align}
where the contour ${\mathcal C} $ is the unit circle with suitable deformations (see \cite{AW,CSSW}) and
the weight function $w(x):=w(x,a,b,c,d; q)$ is defined by
\begin{align*}
w(\cos\theta)=\frac{(e^{2 I\theta}, e^{-2 I\theta};q)_\infty}{h_0\cdot (ae^{I\theta}, ae^{-I\theta},be^{I\theta}, be^{-I\theta},ce^{I\theta}, ce^{-I\theta},de^{I\theta}, de^{-I\theta};q)_\infty}.
\end{align*}
Except the constant factor $C$,  
 the following representation for the Askey-Wilson polynomials in terms of moments was already given by Atakishiyev and Suslov~\cite{AS92}. Their  proof was based on a generalization of Hahn's approach and 
 Wilson~\cite{Wilson} suggested a method to evaluate this determinant  directly. 
\begin{theorem}\label{wilsondet} For $n\geq 1$, $0\leq i\leq n-1$ and $0\leq j\leq n$ let
\begin{align}\label{Gramcoeff}
 A_{i,j}
 = (ac, ad ; q)_i (bc,bd ; q)_j \frac{(ab ; q)_{i+j}}{(abcd ; q)_{i+j}},
\end{align}
and
$ A_{n,j} = (bz, b/z ; q)_j$. Then
\begin{align}\label{Gramdet}
 \det ( A_{i,j} )_{0 \le i, j \le n} =
 C\cdot  p_n(x; a,b,c,d ;q),
\end{align}
where $x=(z+z^{-1})/2$ and 
\begin{align*}
 C= (-1)^n a^{n(n-1)/2} b^{n(n+1)/2} q^{n(n-1)(2n-1)/6}
\prod_{i=0}^{n-1}\frac{(ab,ac,ad,bc,bd,cd,q ; q)_{i}}{(abcd ; q)_{n+i}}.
\end{align*}
\end{theorem}
\begin{proof}
By  \eqref{orth} we have $\oint_{\mathcal C} w(\cos\theta)\frac{dz}{4\pi iz}=1$.
It follows that
\begin{align}\label{coefficients}
\oint_{\mathcal C}  &w(x,a,b,c,d;q)(az, a/z ; q)_i (bz, b/z ; q)_j  \frac{dz}{4\pi iz}\nonumber\\
&=\frac{h_0(aq^i,bq^j,c,d,q)}{h_0(a,b,c,d,q)}\oint_{\mathcal C}  w(x,aq^i, bq^j,c,d;q)\frac{dz}{4\pi iz}\\
&=(ac, ad ; q)_i (bc, bd ; q)_j \frac{(ab ; q)_{i+j}}{ (abcd ; q)_{i+j}},\nonumber
\end{align}
which coincides with  $A_{i,j}$ in \eqref{Gramcoeff} for $i,j=0,\ldots, n$.
Thus, by choosing $ \psi_i = (az, a/z ; q)_i$ and $\phi_j = (bz, b/z ; q)_j$  in \eqref{gramwilson}  we obtain the determinant
in \eqref{Gramdet}. It remains to compute the factor $C$.
As
\begin{align*}
p_{n}(x;a,b,c,d;q)&=2^n(abcdq^{n-1};q)_n x^n+\textrm{ lower terms},\\
 (bz, b/z ; q)_n&=\prod_{k=0}^{n-1}(1-2bxq^k+b^{2}q^{2k}),
 \end{align*}
  comparing the coefficients of $x^n$ in \eqref{gramwilson}  we get
\begin{align}\label{lasteq}
C&=\frac{(-1)^nb^n q^{n(n-1)/2}}{(abcdq^{n-1};q)_n}\det(A_{i,j})_{0\leq i,j\leq n-1}.
\end{align}
The determinant in \eqref{lasteq}  is essentially the Hankel determinant associated to  the moments of little $q$-Jacobi polynomials (see \cite{ITZ09})
\begin{align}\label{littlejacobi}
\det\biggl ( \frac{(ab ; q)_{i+j}}{ (abcd ; q)_{i+j}}  \biggr)_{0\leq i,j\leq n-1}=(ab)^{\frac{n(n-1)}{2}}q^{\frac{n(n-1)(n-2)}{3}}\prod_{k=0}^{n-1} \frac{(q,ab,cd;q)_{k}}{(abcd;q)_{k+n-1}},
\end{align}
which  is also a special case of \eqref{eq:ITZ1}.
It follows from \eqref{littlejacobi} that 
\begin{multline}\label{littlejacobibis}
\det\biggl ((ac, ad ; q)_i (bc,bd ; q)_j \frac{(ab ; q)_{i+j}}{ (abcd ; q)_{i+j}}  \biggr)_{0\leq i,j\leq n-1}\\=(ab)^{\frac{n(n-1)}{2}}q^{\frac{n(n-1)(n-2)}{3}}
\prod_{j=1}^{n-1}(ac, ad,bc,bd ; q)_j
\prod_{k=0}^{n-1} \frac{(q,ab,cd;q)_{k}}{(abcd;q)_{k+n-1}}.
\end{multline}
Substituting this in \eqref{lasteq}  we recover the factor $C$ in formula~\eqref{Gramdet}.
\end{proof}
\begin{remark}
Using \eqref{littlejacobibis}, 
formula~\eqref{Gramdet} can also be proven from the Desnanot-Jacobi adjoint matrix theorem
and a special case of the known contiguous relation~(see \cite[(3.3)]{Kratt93})
\begin{multline}
{}_r\phi_s
\left[
\begin{matrix}
aq,\A\cr
bq,\B\cr
\end{matrix};
q,z
\right]
-
{}_r\phi_s
\left[
\begin{matrix}
a,\A\cr
b,\B\cr
\end{matrix};
q,z
\right]
\\
=
\frac{
(-1)^{1+s-r}z(a-b)
}
{
(1-b)(1-bq)
}
\frac{
\prod_{i=1}^{r-1}(1-A_i)
}
{\prod_{i=1}^{s-1}(1-B_i)
}
{}_r\phi_s
\left[
\begin{matrix}
aq,\A q\cr
bq^2,\B q\cr
\end{matrix};
q,q^{1+s-r}z
\right],
\end{multline}
where $\A=(A_1, \ldots, A_{r-1})$ and $\B=(B_1, \ldots, B_{s-1})$ for $r,s\geq 2$.
\end{remark}

Actually  it is not hard to see that  Theorems~\ref{newthm}  and  \ref{wilsondet} are equivalent.
\begin{proposition}  
The two determinant formulae \eqref{Gramdet}  are  \eqref{eq:det} equivalent.
\end{proposition}
\begin{proof}
Let  $z = e^{I \theta}$ and 
 $x = \cos \theta$. Consider the matrix $A:=( A_{i,j} )_{0 \le i, j \le n}$, where $A_{i,j}$  are  given  in 
 \eqref{Gramcoeff}.
Upon multiplying the $(j-1)$st column of $A$ by $1 - 2 bx q^{j-1} + b^2 q^{2j-2}$
and subtracting from the $j$th column, for $j = n, n-1, ..., 1$, the
last row of the matrix becomes $(1, 0, ..., 0)$, and
$$
A_{i,j}-(1 - 2 bx q^{j-1} + b^2 q^{2j-2}) A_{i,j-1}=(ac,ad;q)_i(bc,bd;q)_{j-1} B_{i,j},
$$
where $B_{i,j}$  is given in \eqref{defB} for $0\leq i\leq n-1$ and $1\leq j\leq n$.
Hence,
the two formulae \eqref{Gramdet}  and   \eqref{eq:det} are  equivalent.
\end{proof}

By \eqref{orth}  the linear functional  $\L: \C[x]\mapsto \C$  associated to the orthogonal  measure of  the Askey-Wilson polynomials
has the explicit integral expression:
\begin{align}\label{LF}
\L\bigl(x^n\bigr)=\oint_{\mathcal C}\left(\frac{z+z^{-1}}{2}\right)^n w(\cos\theta)\frac{dz}{4\pi iz}.
\end{align}
It follows from \eqref{coefficients}  with $j=0$ and $i=n$ that 
 \begin{align}\label{linfunc}
\L\bigl((az,a/z;q)_n\bigr)=\frac{(ab,ac,ad;q)_n}{(abcd;q)_n}.
\end{align}
Clearly, if we take $\psi_i(x)=\phi_i(x)=x^i$ for $i\geq 0$  in \eqref{gramwilson}, then $\mu_{i,j}=\L\bigl(x^{i+j}\bigr)$ are the moments of Askey-Wilson polynomials 
and we obtain another determinant expression for the Askey-Wilson polynomials.
Such a formula would be interesting if we have a simple formula for the moments~\eqref{LF}.
Recently Corteel et al. \cite{CSSW}  and Ismail-Rahman~\cite{IS} have published  a double-sum formula for the moments  of Askey-Wilson polynomials.
We  would like to point out    that their  formula does follow  straightforwardly  from \eqref{linfunc} and 
 the Newton interpolation formula (see \cite[Chapter~1]{MT}), that we recall below.
\begin{theorem}[Newton's interpolation formula] Let  $b_0, b_1, \ldots, b_{n-1}$ be distinct complex numbers.
Then, for any polynomial $f$ of degree less than or equal to $n$ we have
\begin{align}\label{newton}
f(x)=\sum_{k=0}^n \Biggl( \sum_{j=0}^k\frac{f(b_j)}
{\prod_{r=0, r\neq j}^{k} (b_j-b_r)} \Biggr) (x-b_0)\cdots (x-b_{k-1}).
\end{align}
\end{theorem}

Indeed, if
$b_j=(q^{-j}/a+aq^j)/2$ for $j=0,\ldots, n-1$, then
\begin{align*}
\prod_{r=0}^{j-1}(b_j-b_r)&=(-1)^j2^{-j}a^j q^{j\choose 2}(q,q^{-2j+1}/a^2;q)_j,\\
\prod_{r=j+1}^{k}(b_j-b_l)&=(-1)^{k-j}2^{j-k} a^{j-k}
q^{-(j+1)(k-j)-{k-j\choose 2}}(q,a^2q^{2j+1};q)_{k-j}.
\end{align*}
As   $x=(z+1/z)/2$ we have
$$
(x-b_0)\cdots (x-b_{k-1})=(-1)^k2^{-k}a^{-k}q^{-{k\choose 2}}(az, a/z;q)_k.
$$
Substituting the above into
 \eqref{newton} we obtain Proposition~3.1 of  \cite{CSSW}, i.e.,
\begin{align}\label{newtonspecial}
f(x)=\sum_{k=0}^n(az, a/z;q)_k  \sum_{j=0}^k
\frac{q^{k-j^2}a^{-2j} f\bigl(\frac{q^{j}a+q^{-j}/a}{2}\bigr)}
{(q,q^{-2j+1}/a^2; q)_{j} (q,q^{2j+1}a^2;q)_{k-j}}.
\end{align}
 In particular,    
 applying the linear functional  $\mathcal L$ to the two  sides of  \eqref{newtonspecial} with $f(x)=(t+x)^n$
 we obtain  Theorem~1.13 of \cite{CSSW} with a typo corrected, which also appears  in \cite{IR}.
\begin{theorem}[Corteel-Stanley-Stanton-Williams, Ismail-Rahman] For any fixed $t\in \C$,
we have
\begin{align}\label{eq:mom}
\L((t+x)^n)=\sum_{k=0}^n \frac{(ac,ab,ad;q)_k}{(abcd;q)_k}\sum_{j=0}^k
\frac{q^{k-j^2}a^{-2j} \bigl(t+\frac{q^{j}a+q^{-j}/a}{2}\bigr)^n}
{(q,q^{-2j+1}/a^2; q)_{j} (q,q^{2j+1}a^2;q)_{k-j}}.
\end{align}
\end{theorem}
\begin{remark} 
The proof of the above formula in \cite{CSSW, IR} was based on a result of Ismail and Stanton~\cite[Theorem~3.3]{CSSW}, which is
 equivalent to \eqref{newtonspecial}.   Since Ismail and Stanton's formula  was originally proved using
the Askey-Wilson operator in \cite[Theorem~20]{IS}, our proof seems more accessible for people who are not familiar with Askey-Wilson operator.
The $t=0$  of  \eqref{eq:mom} is the starting point of \cite{KS}.
 \end{remark}

Finally, it is interesting to note that  
the following special case  of Newton's formula \eqref{newton} has been rediscovered recently by  Mansour et al.~\cite{MSS}.

\begin{cor} We have
\begin{align}\label{man}
(x+a_0)\cdots (x+a_{n-1})=\sum_{k=0}^nu(n,k) (x-b_0)\cdots (x-b_{k-1}),
\end{align}
where
$$
u(n,k)=\sum_{r=0}^k\frac{\prod_{j=0}^{n-1}(b_r+a_j)}{\prod_{j=0, j\neq r}^{k} (b_r-b_j)} \quad \text{for} \quad k=0,\ldots, n.
$$
\end{cor}
\begin{remark}
Clearly the connection coefficients  $u(n,k)$ in \eqref{man} are characterized by
 the recurrence
 \begin{align*}
  u(n,k)=u(n-1,k-1)+(a_{n-1}+b_k)u(n-1,k)
  \end{align*}
 with boundary conditions $ u(n,0)=\prod_{i=0}^{n-1}(a_i+b_0) $ and $u(0,k)=\delta_{0,k}$ (the Kronecker delta function).
 Hence we recover  the main result of Mansour et al.~\cite{MSS} (see also \cite{Simpson}).
 \end{remark}
\section*{Acknowledgements}
This paper was presented at 
the 25th International Conference on Formal Power Series and Algebraic Combinatorics (FPSAC'13), 
Paris, France, June 24--28, 2013.
We are grateful to an anonymous referee  for  the conference FPSAC'13 for pointing out the link to Gram determinants of Askey-Wilson polynomials and to Sergei Suslov for bringing the reference~\cite{AS92} to our attention.

\bibliographystyle{amsplain}

\end{document}